\documentclass{amsproc}
\usepackage{graphicx}
\usepackage{latexsym}
\usepackage{amsfonts}
\usepackage[all]{xy}

\usepackage{amssymb, mathrsfs, amsfonts, amsmath}
\usepackage{amsbsy}
\usepackage{amsfonts}
\usepackage{amssymb}
\newtheorem{theorem}{Theorem}[section]
\newtheorem{lemma}[theorem]{Lemma}
\newtheorem{corollary}[theorem]{Corollary}
\newtheorem{proposition}[theorem]{Proposition}

\theoremstyle{definition}
\newtheorem{definition}[theorem]{Definition}
\newtheorem{example}[theorem]{Example}

\theoremstyle{remark}
\newtheorem{remark}[theorem]{Remark}

\numberwithin{equation}{section}

\begin{document}
\title[A DIVERGENCE THEOREM FOR NON-COMPACT RIEMANNIAN MANIFOLDS]{A DIVERGENCE THEOREM FOR NON-COMPACT RIEMANNIAN MANIFOLDS: A DYNAMICAL APPROACH}
  
  %\thanks{Version 2.0, 1999/11/15}}
\author{\'Italo Melo}
\address{Departamento de
	Matem\'{a}tica-UFPI, Ininga, 
	64049-550 Teresina, Brazil}
%\thanks{$^{1}$ Partially supported by CAPES-BR}
\email{italodowell@ufpi.edu.br}

\author[Enrique Pujals]{Enrique Pujals}
\address{Instituto Nacional de Matem\'atica Pura e Aplicada\\
Estrada Dona Castorina 110, 22460-320 Rio de Janeiro, Brazil}
%\thanks{$^{2}$ Partially supported by CNPq-BR}
\email{enrique@impa.br}
%\thanks{Even e-mail addresses can have
 %   footnotes!}} 
%\author[Magnes Press]{The Hebrew University Magnes Press\thanks{This
 %   is the copyright owner of the style}}
%\author[IJM]{Israel Journal of Mathematics Editorial Board\thanks{This
 % entry is inserted just to show how to typeset several authors with
  %the same address}}
%\address{P.O. Box 39099\\ Jerusalem 91390\\ 
%Israel. O. Box 6248\\
%Providence, RI 02940\\
% USA

%\date{Received on MONTH, YEAR}
%\issueinfo{VOL}{NUM}{MONTH}{YEAR}
%\doiinfo{10.1007/DOI-NUMBER}
\begin{abstract}
In this paper we use a dynamical approach to prove some new divergence theorems on complete non-compact Riemannian manifolds.
\end{abstract}
\maketitle
%\tableofcontents

\section{Introduction and statement of results}

A divergence theorem states that $    \int_{M}  (\rm{div} X) \, d\nu_{g}    =0$, under certain assumptions on $X$ and $M$, where $M$ is a Riemannian manifold, $X$ is a vector field on $M$ and $\rm{div} X$ denotes the divergence
of $X$. The starting point is the usual divergence theorem for the case where $X$ is smooth and has compact support.

On a closed Riemannian manifold the classical divergence theorem is a very useful tool in the study of PDEs. In particular, in the study of  differential operators in divergence form, such as the Laplace-Beltrami operator, the $p$-Laplace operator and the mean curvature operator.

In this context, the divergence theorem can be used to obtain an integration by parts that facilitates the obtaining of integral  estimates of solutions of PDEs. Furthermore, it is possible to use the divergence theorem to obtain comparison results and deduce information of the geometry and the topology of this Riemannian manifold.
% The divergence theorem is used to obtain comparison theories and  integral estimates 

If a divergence theorem is available for a complete non-compact Riemannian manifold then it would be  interesting to generalize some of the existing results for compact Riemannian manifolds.

In that direction, we present a new divergence theorem for the non-compact case, with or without finite volume, based on the dynamical behavior of the geodesic flow and the integrability of a functional associated to the vector field. To precise, let us introduce some definitions.

Let $(M,g$) be a complete Riemannian manifold without boundary, we denote by $\nabla$ the Levi-Civita connection. Given a $C^1$ vector field on $M$ we can consider the continuous function $f_X$ on the unit tangent bundle $SM$ defined by 
$f_X(p,v) = g(\nabla_v X,v)$.

In the theorem below we present a new divergence  theorem on complete Riemannian manifolds without boundary whose geodesic flow is recurrent.
%In this section we give a dynamical proof of the Theorem 4 for the case where M has finite volume.
%and X a $C^1$ vector field on M. Consider the continuous function $f_X$ on SM defined by 
%$f_X(p,v) = g(\nabla_v X,v)$ and the Riemannian submersion $\pi$

\begin{theorem}\label{infinito}
	Let $M$ be a non-compact complete Riemannian manifold without boundary and $X$ a $C^1$ vector field. If the geodesic flow $\phi_{t}$ is recurrent with respect to the Liouville measure and  $f_X$ is integrable on $SM$ then ${\rm div}\, X$ is integrable on $M$ and
	$$  \displaystyle\int_{M}  {\rm div}\, X \, d\nu_{g}    =0. $$
\end{theorem}

In particular, if $M$ has finite volume we get the following corollary.

\begin{corollary}\label{finito}
	Let M be a non-compact complete Riemannian manifold without boundary  with finite volume
%In this section we give a dynamical proof of the Theorem 4 for the case where M has finite volume.
and $X$ a  $C^1$ vector field on M. If $f_X$ is integrable on SM then ${\rm div}\, X$ is integrable on $M$ and

$$  \displaystyle\int_{M} {\rm div}\, X \, d\nu_{g}    =0.$$
\end{corollary}

If $(M,g)$ is a non-compact Riemannian manifold whose geodesic flow is not necessarily recurrent then using the E. Hopf's decomposition and an additional hypothesis we get the following result.

\begin{theorem}\label{infinito2}
	Let $M$ be a non-compact complete Riemannian manifold without boundary and $X$ a $C^1$ vector field. If $f_X$ is integrable on $SM$ and $|X| \rightarrow 0$ uniformly at infinity in $M$ then ${\rm div}\, X$ is integrable on $M$ and
	$$  \displaystyle\int_{M}  {\rm div}\, X \, d\nu_{g}    =0. $$
\end{theorem}

The essential tool in the proofs is the Maximal Ergodic Theorem (see section \ref{ergodic sec}); similar approach was already used in \cite{flo:gui} and \cite{yau:yau}.

 It is worth mentioning that there exist Riemannian manifolds with infinite volume whose geodesic flow is recurrent with respect to Liouville measure; an example of this type of manifold is the hyperbolic surface of divergence type.

	A hyperbolic surface is a complete, two dimensional Riemannian manifold with constant curvature $-1$. This surface has the unit disc as universal cover and can be viewed as $  H / \Gamma $ where $H$ is the unit disc equipped with the hyperbolic metric and $\Gamma$ is the covering group of isometries of $H$. 

In 1939, Hopf, proved that geodesic flows on hyperbolic surfaces of infinite area are either totally dissipative or recurrent and ergodic. Nicholls (see \cite{p:nichols}) proved that the geodesic flow on the hyperbolic surface $H / \Gamma$ is recurrent and ergodic if and only if the Poincar\'e series of $\Gamma$ diverges at $s=1$, this result was also proved by Aaronson and Sullivan in \cite{a:s}. The hyperbolic surfaces whose Poincar\'e series of $\Gamma$ diverges at $s=1$ are called of divergence type.

Other type of examples go as follows: given a  closed subset  $\Lambda$ of the Riemann sphere $\hat{\mathbb{C}}$ containing at least 3 points; it is known that the Riemann surface $M = \hat{\mathbb{C}} \backslash \Lambda$ has universal cover $\mathbb{H}^2$, where $\mathbb{H}^2$ is the hyperbolic space and therefore admits a hyperbolic metric.  One knows that the hyperbolic area of $M$ is infinite once $\Lambda$ is infinite. Furthermore, if $\Lambda$ has logarithmic capacity zero for every positive measure $\mu$ on $\Lambda$ then the geodesic flow on $M$ is ergodic and hence $M$ is a hyperbolic surface of divergence type. For more details see \cite{a:s}. 

So we have many examples of hyperbolic surfaces with infinite volume whose geodesic flow is ergodic and recurrent.

In \cite{a:d}, Denker and Aaronson also construct other examples of Riemannian manifolds with infinite volume whose geodesic flow is recurrent.

For the case of zero curvature, the geodesic flow associated to  ``infinite staircase’ square tiled surface (which provide a non-compact surface with infinite volume and planar curvature)  it is also recurrent (see for instance \cite{hhw:HHW}).

Similar  type of divergences results were already obtained. 
In \cite{dihe:newdir}, Gaffney proved that if $X$ is a $C^1$ vector field such that $|X| \in L^1(M)$ and  ${\rm div}\, X$ has an integral (i.e either ($    {\rm div}\, X)^{+}$ or  ($ {\rm div}\, X )^{-}$ is integrable) then $    \int_{M} {\rm div}\, X d\nu_{g}    =0    $. This result was generalized by Karp in \cite{l:karp} with a weaker hypothesis about the norm of the vector field $X$. Similar results for the non-complete case were also obtained for $p-$parabolic Riemmanian manifolds (see details in section \ref{classical div}).

The main difference of the above mentioned results with   our  new divergence theorems is that  we assume an integrability condition on the function $f_X$ but not necessarily in the norm of the vector field, which is the common hypothesis in the results proved somewhere else.

In fact, using either surfaces of revolution or the classical construction of warped product, we prove in section \ref{example} that there are  examples (with finite and infinite volume) that satisfy that $f_X$ is integrable  but $|X|$ is not integrable. For that it is used proposition \ref{suficiente} (see section \ref{sec-suficiente}) that  provides  a sufficient condition under which the function $f_X$ is integrable (it is useful to compare the statement of above mentioned proposition with the hypothesis of the theorem due by Karp (see theorem \ref{teokarp}).

At the end of the paper, we provide some applications to potential theory.

 %The main tool for the proof of Theorem \ref{infinito} is the Maximal Ergodic Theorem.

\section{Notation and Preliminaries}
In this section we present the notations and some preliminaries results that will be used in the sequel. Throughout the rest of the
paper, $(M,g)$ will denote a Riemannian manifold without boundary with dimension $n \geq 2$.
\subsection{Fubini's Theorem for  the case of Riemannian submersions}
Let $\nu_g$ be the Riemannian measure or volume measure associated to Riemannian manifold $(M,g)$. Now consider a Riemannian submersion $\pi:(M,g) \rightarrow (N,h)$. Recall that for each $q \in \pi(M)$, $\pi^{-1}(q)$ is an $(n-m)$-dimensional submanifold, which carries the Riemannian measure $\nu_{g_q}$ with respect to the metric $g_q$ induced on $\pi^{-1}(q)$ from $g$. For a function $f$ defined over $M$, we set $f_q $ the restriction of $f$ to  $\pi^{-1}(q)$. Now if $f_q$ is an integrable function on $\pi^{-1}(q)$ with  respect to $\nu_{g_q}$, define 
$$\overline{f}(q) = \displaystyle\int_{\pi^{-1}(q)} f_q \, d\nu_{g_q}.$$

We conclude the subsection stating a result that will be used in the next section.

\begin{theorem} \label{FU}(Theorem 5.6, p. 66, \,\cite{sakay:sakay}) Let $\pi:(M,g) \rightarrow (N,h)$ be a surjective Riemannian submersion. If f is a real-valued continuous function with compact support (resp., an integrable function ) on M, then $\overline{f}$ is a continuous function with compact support  on N (resp., $f_q$ is integrable for almost all $q \in N$ and $\overline{f}$ is a integrable function on N ), and
	$$   \displaystyle\int_{M} f \, d\nu_g =  \displaystyle\int_{N} \overline{f} \, d\nu_h.$$

For more details see \cite{sakay:sakay}.	
\end{theorem}
\subsection{Geodesic flow }
%(see \cite{GP})
Let $(M,g$) be a complete Riemannian manifold. We denote by $\gamma_{\theta}(t)$ the unique geodesic with initial conditions 
$\gamma_{\theta}(0)=p$ and $\gamma_{\theta}'(0)=v$, where $\theta=(p,v)$ is a point in the tangent bundle $TM$. For a given $t\in \mathbb{R}$, we 
define the following diffeomorphism, $\phi_t:TM \to TM $ with  $\phi_t(\theta)=(\gamma_{\theta}(t),\gamma_{\theta}'(t))$.

%\begin{eqnarray*}

%\phi_t&:&TM \to TM  
%\theta=(p,v) \to (\gamma_{\theta}(t),\gamma_{\theta}^{\prime}(t))
%\end{eqnarray*}
%as follows $\phi_t(\theta)=(\gamma_{\theta}(t),\gamma_{\theta}'(t))$. The family of
The family of diffeomorphism $\phi_t$ is in fact a flow (called \textit{geodesic flow}) once
it satisfies $\phi_{t+s}=\phi_t\circ\phi_s$. %The vector field associated 

Let $SM$ be the unit tangent bundle of $M$, i.e, the subset of $TM$ given by those 
pairs $\theta=(p,v)$ such that $v$ has norm one. Since geodesics travel with constant speed, 
we see that $\phi_t$ leaves $SM$ invariant, that is, given $\theta\in SM$ then for 
every $t\in\mathbb{R}$ we have that $\phi_t(\theta)\in SM$.
% The restriction of $\phi_t$ to SM is called the geodesic flow of g.

Let $\pi:TM \rightarrow M$ be the canonical projection, where $\pi(\theta) = p$. 
\begin{definition}
	There exists a canonical subbundle of $TTM$ called the vertical subbundle whose fiber at $\theta$ is given by $V(\theta) = \ker(d_{\theta}\pi)$.
\end{definition}
%Geometrically, $V(\theta)$ is the tangent space to the fiber $T_pM \subset TM$ at the point $\theta$. We denote by  $\nabla$ the Levi-Civita connection on M and by $\displaystyle\frac{DV}{dt}$ the derivative covariant of vector field $V(t)$ along the curve $x(t)$.

We shall define the connection map $K: TTM \rightarrow TM$ as follows. Let $\xi \in T_{\theta}TM$ and $z:(-\epsilon,\epsilon) \rightarrow TM$  be a curve adapted to $\xi$, i.e, $z(0) = \theta$ and $z'(0) = \xi$, where 
$z(t) = (\alpha(t),Z(t))$,

%Such a curve can be seen as a curve $\alpha:(-\epsilon,\epsilon) \rightarrow TM$, $\alpha = \pi \circ z$, and a vector field Z along $\alpha$, equivalently, $z(t) = (\alpha(t),Z(t))$. Define

$$K_{\theta}(\xi) = \displaystyle\frac{DZ}{dt}(0).$$

The map $K_{\theta}$ is well defined and is linear for each $\theta$.
\begin{definition}
	The horizontal subbundle is the subbundle of $TTM$ whose fiber at $\theta$ is given by $H(\theta) = \ker (K_\theta)$.
\end{definition}
For each $\theta$, the maps $d_{\theta}\pi|_{H(\theta)}: H(\theta) \rightarrow T_pM$ and $K_{\theta}|_{V(\theta)}:V(\theta) \rightarrow T_pM$ are linear isomorphisms. Furthermore, $T_{\theta}TM = H(\theta) \oplus V(\theta)$ and the map \linebreak $j_{\theta}:T_{\theta}TM \rightarrow T_pM \times T_pM$ given by 
$$ j_{\theta}(\xi) = (d_{\theta}\pi(\xi),K_{\theta}(\xi)),         $$
is a linear isomorphism. 
%From now on, whenever we write $\xi=(\xi_h,\xi_v)$ we mean that we identify $\xi$ with $j_{\theta}(\xi)$, where $\xi_h =d_{\theta}\pi(\xi) $ and $\xi_v = K_{\theta}(\xi)$.
\begin{definition}
	Using the decomposition $T_{\theta}TM = H(\theta) \oplus V(\theta)$, we can define in a natural way  a Riemannian metric on $TM$ that makes $H(\theta)$ and $V(\theta)$ orthogonal. This metric is called the Sasaki metric and is given by
	$$ g_{\theta}^S(\xi,\eta) = g_{\pi(\theta)}(d_{\theta}\pi(\xi), d_{\theta}\pi(\eta)  )  +  g_{\pi(\theta)}(K_{\theta}(\xi),    K_{\theta}(\eta)       ) $$
\end{definition}

%The geodesic vector field $G:TM \rightarrow TTM$ is given by 
%$$G(\theta) = \displaystyle\frac{\partial}{\partial t}\Big|_{t=0} \phi_t(\theta) =\displaystyle\frac{\partial}{\partial t}\Big|_{t=0} (\gamma_{\theta}(t),\gamma'_{\theta}(t)), $$

%Therefore, $G(\theta) = (v,0)$ using the identification $j_{\theta}$. 

Now consider the Sasaki metric restricted to the unit tangent bundle $SM$. The projection $\pi:SM \rightarrow M$ is a surjective Riemannian submersion. Furthermore, the geodesic flow in $SM$ preserves the Riemannian measure $\nu_{g^S}$ induced by the Sasaki metric. This measure coincides with the Liouville measure up to a constant. When $M$ has finite volume the Liouville measure is finite. For more details see \cite{gabriel:paternain}.
% A vector $\xi \in T_{\theta}TM$ lies $T_{\theta}SM$ if and only if $g_{p}(K_{\theta}(\xi),v) = 0$, in particular $G(\theta) \in T_{\theta}SM$. Furthermore, $\pi:SM \rightarrow M$ is a surjective Riemannian submersion.

%The geodesic flow in SM preserves the Riemannian measure $\nu_{g^S}$ induced of the Sasaki metric, this measure coincides with the Liouville measure. When M has finite volume the Liouville measure is finite, for more details see \cite{gabriel:paternain}.
%The Riemannian measure 

%Let $\alpha$ be a one-form of SM defined as $\alpha_{\theta}(\xi) = g_{\theta}^S(\xi,G(\theta)) = g_{p}(d_{\theta}\pi(\xi),v)$.

%\begin{proposition}(Corollary 1.31 ,\,\cite{GP})
%The geodesic flow in SM preserves the one-form $\alpha$.
%\end{proposition}

% The volume form on SM induced by the Sasaki metric coincides (up to sign) with $\displaystyle\frac{1}{(n-1)!}\alpha \wedge (d\alpha)^{m-1}$, where $m = \dim M$. Furthemore the geodesic flow in SM preserves the Liouville measure.

% When M has finite volume we have that $\mu_l = \displaystyle\frac{1}{\nu_{g^S}(SM)} \nu_{g^S}$, in particular the geodesic flow in SM preserves the Riemannian measure $\nu_{g^S}$.
\subsection{Warped products}
%(see \cite{ON})
Let $B,F$ be Riemannian manifolds, with metrics $g_B$ and $g_F$, respectively and $f  >  0$  be a smooth function on $B$. The  warped product  $M  =  B  \times_{f}  F$
is the product manifold  $B  \times  F$  furnished with the Riemannian  metric 
$$    g = \pi_B^{*}( g_B) + (f \circ \pi_B)^2  \pi_F^{*} (g_F),                             $$
where    $\pi_B$  and  $\pi_F$  are the projections  of  $B  \times  F$  onto $B$ and  $F$, respectively.

Let $X$ be  a vector field on $B$. The  horizontal lift of $X$ to $B\times_f F$ is the vector field $\overline{X}$ such that $d{\pi_B}_{(p,q)}(\overline{X}(p,q)) = X(p)$ and $d{\pi_F}_{(p,q)}(\overline{X}(p,q)) = 0$. If $Y$ is a vector field on $F$, the vertical lift of $Y$ to $B\times F$ is the vector field $\overline{Y}$ such that $d{\pi_B}_{(p,q)}(\overline{Y}(p,q)) = 0$ and $d{\pi_F}_{(p,q)}(\overline{Y}(p,q)) = Y(q)$. The set of all such lifts are denoted as usual by $\mathcal{L}(B)$ and $\mathcal{L}(F)$, respectively.

We denote by  $\nabla$,$\nabla^B$ and $\nabla^{F}$ the Levi-Civita connections on $B \times_f F$, $B$ and $F$, respectively.

\begin{proposition}\label{conexao}(Proposition 35, p. 206, \,\cite{o:neil})
	On  $M  =  B  \times_f  F$,  if  $\overline{X}, \overline{Y} \in \mathcal{L}(B)$ and  $\overline{U}, \overline{V} \in \mathcal{L}(F)$ then
	\begin{enumerate}
		\item  $\nabla_{\overline{X}}\overline{Y} = \overline{\nabla_X^B Y},   $\\
		\item  $\nabla_{\overline{U}}\overline{X} = (Xf /f)\overline{U},$\\
		\item $d\pi_B( \nabla_{\overline{U}}\overline{V})= -(g(U,V)/f) \cdot {\rm grad}\, f$,\\
		\item $d\pi_{F}( \nabla_{\overline{U}}\overline{V} ) = \nabla_{U}^{F} V.$
		\end{enumerate}
\end{proposition}
\subsection{The Divergence theorem}\label{classical div}
%(see \cite{GF},\cite{LP},\cite{SA},\cite{JL})
In this subsection we present some results that can be found in the literature on divergence theorems. 

%Let $\mathfrak{X}(M)$ be  the set of smooth vector fields on M.
Let $X$ be a $C^1$ vector field on a Riemannian manifold $(M,g)$.
\begin{definition}
	The differential of $X$ is the linear operator  $A_X:\mathfrak{X}(M)
	\to\mathfrak{X}(M)$, given by $A_X(Y):=\nabla_YX$. Then, to each point $p\in M$, we
	assign the linear map $A_X(p):T_pM\to T_pM$ defined by $A_X(p)v=\nabla_vX$. The divergence
	of $X$, denoted by ${\rm div}\,X$, is the trace of this differential. 
\end{definition}
%\begin{thm}(Corollary 16.13 ,\,\cite{JL})
%	If M is a compact oriented smooth manifold without boundary, then the integral of every exact n-form over M is zero:
%	$$\displaystyle\int_{M} d\omega = 0$$
%\end{thm}
%Let  $(M,g)$ be a compact oriented Riemannian manifold without boundary. If $\omega$ is a $(n-1)$-form in M given by $\omega =
%\mathfrak{i}_X dM$, i.e, the contraction of $dM$ in the direction of a smooth vector field X on M then $d\omega = (\rm{div} X) dM$ thus
%$$  \displaystyle\int_{M} d\omega = \displaystyle\int_{M} (\rm{div}X) dM = \displaystyle\int_{M}  (\rm{div} X) d\nu_{g}    =0 .  $$
%In the general case, we have

If $X$ has compact support and $M$ is a Riemannian manifold then from  classic divergence theorem it follows that
	$$  \displaystyle\int_{M}  {\rm div}\, X \, d\nu_{g}    =0.$$

%\begin{thm}\label{classico}(The classical Divergence Theorem)(Theorem 5.11,\,\cite{SA}) Let X be a $C^1$ vector field with compact support on a  Riemannian manifold without boundary M. Then
%	$$  \displaystyle\int_{M}  (\rm{div} X) d\nu_{g}    =0 $$
	
%\end{thm}

Gaffney in \cite{dihe:newdir}, extended this result to complete Riemannian manifolds $M$ by proving that, given a $C^1$ vector field $X$
on $M$, we have 
$$  \displaystyle\int_{M}  {\rm div}\, X \, d\nu_{g}    =0, $$
provided $|X|$ and ${\rm div}\, X $ are integrable (but in fact, if
$({\rm div}\, X)^{-} =   \max \{{\rm div}\, X,0\}$ is integrable then the result remains true). This result was later extended by Karp. He proved the following result,
\begin{theorem}(\cite{l:karp})\label{teokarp}
	Let $M^n$ be a complete non-compact Riemannian manifold without boundary and X a vector field such that
	\begin{eqnarray*}\label{eqkarp}
	\displaystyle\liminf_{r\rightarrow +\infty} \displaystyle\frac{1}{r} \displaystyle\int_{  B(2r)/B(r) }       |X| \, d\nu_g &=&0. \\              
	\end{eqnarray*} 
	If ${\rm div}\, X$ has an integral 
	%(i.e either ($\rm{div}X)^{+}$ or  ($\rm{div}X)^{-}$ is integrable) then 
then	$$  \displaystyle\int_{M}  {\rm div}\, X \, d\nu_{g}    =0. $$
	In particular, if outside of some compact set ${\rm div}\, X$ is  everywhere $\geq0$ (or $\leq 0$) then $  \displaystyle\int_{M}  {\rm div}\, X \, d\nu_{g}    =0$.
\end{theorem}
\begin{remark}
	In the previous theorem, $B(r)$ denotes the geodesic ball of radius $r$ at some point $p \in M$.
\end{remark}

To consider the non-necessary complete manifolds let us introduce the notion of $p-$parabolicity.

\begin{definition}
	A Riemannian manifold M is said to be $p$-parabolic, \linebreak $1 < p < \infty$, if every solution $u \in W_{loc}^{1,p}(M) \cap C^{0}(M)$ of the problem
	$$
	\left\{
	\begin{array}{ccccc}
	\Delta_p u & \geq & 0 & on & $M,$ \\
	\sup_M u &<& +\infty \\
	
	\end{array}
	\right.
	$$
	must be constant. A manifold that is not $p$-parabolic will be called \linebreak $p$-hyperbolic.
\end{definition}
There are several equivalent definitions of this notion, for more details see \cite{pi:setti}, \cite{tro:tro}. In the definition above $W_{loc}^{1,p}(M)$ denotes the space of the local $L^p$-functions such that for every compact $K$ the distributional gradient are in $L^p(K)$ and $\Delta_p u$ denotes the non-linear operator defined by,
$$ \Delta_p u = {\rm       div}\, \big( |{\rm grad}\, u|^{p-2}  \cdot {\rm grad}\, u   \big)                   .$$
This operator is usually called $p$-Laplace operator. When $p=2$ this operator coincides with the Laplace-Beltrami operator.

%The next result provides a geometric condition to ensure the $p$-parabolicity, a proof this result can be found in \cite{PS}.
%\begin{prop}
%	Let $M$ be a complete Riemannian manifold. If
%	$$      \displaystyle\int^{+\infty} \displaystyle\frac{dt}{v(t)} = + \infty,     $$
%	where $v(t) = {\rm{area}(\partial B_t(o))}^{1/p-1}$, then $M $ is $p$-parabolic.
%\end{prop}

A complete Riemannian manifold with finite volume is $p$-parabolic for all $1 < p < \infty$. A complete Riemannian manifold with polynomial growth of degree d is $p$-parabolic for all $p \geq d$. For example a complete n-dimensional Riemannian manifold with non negative Ricci curvature is $p$-parabolic for all $ p \geq n$, in particular the Euclidean space $\mathbb{R}^n$ is $p$-parabolic for all $p \geq n$.

A complete simply connected Riemannian manifold with sectional curvature $K \leq -1$ is $p$-hyperbolic for all $1 < p< \infty$, in particular the hyperbolic space $\mathbb{H}^n$ is $p$-hyperbolic for all $1 < p< \infty$.

Let $M$ be a Riemannian manifold 2-parabolic, not necessarily complete. From the results of  Lyons and Sullivan in \cite{ly:sul}, it follows that if a vector field $X$ satisfies $|X| \in L^{1}(M)$, ${\rm div}\, X \in L_{\rm{loc}}^1(M)$ and $  ({\rm div}\, X)^{-}$ is integrable then $    \int_{M} {\rm div}\, X d\nu_{g}    =0        $. In the general case, from the results of Gol'dshtein and Troyanov in \cite{go:tro} with the assumption that  $|X| \in L^{\frac{p}{p-1}}$, it follows that this result remains true for $p$-parabolic Riemannian manifolds with $p>1$ thus gives us a $p$-parabolic analogue of the Gaffney result.
Moreover, also  in \cite{go:tro} it is  proved that a Riemannian manifold $M$ is $p$-parabolic if, and only if, the following result holds. Let $X$ be a vector field satisfying $|X| \in L^{\frac{p}{p-1}}$, ${\rm div}\, X \in L_{loc}^1(M)$ and ${\rm div}\, X$ has an integral. Then 
$$       \displaystyle\int_{M} {\rm div}\, X  \, d\nu_{g}    = 0 .               $$

In \cite{val:ve} Valtorta and Veronelli present some new Stokes's theorem on complete manifolds that extends, in different directions, previous works of Gaffney and Karp and also the so called Kelvin-Nevanlinna-Royden criterion for $p$-parabolicity.

\subsection{ Ergodic Theory}\label{ergodic sec}
%(see \cite{BS}, \cite{PW}, \cite{PT}, \cite{AS},\cite{Z}). 
We say that $(X, \mathcal{B})$ is a standard measurable space if $X$ is a Polish space  and $\mathcal{B}$ is the Borel $\sigma$-algebra.
Let $(X,\mathcal{B},m)$ be standard measure space, that is, $(X,\mathcal{B})$ is a standard measurable space. We say that $A=B \mod m$ if $m((A \backslash B) \cup (B \backslash A)) = 0$.

%Given $A \in \mathcal{B}$, let $ \mathcal{B} \cap A = \{B \in \mathcal{B}: B \subset A\}$, in \cite{AS} the author proves that $(A,\mathcal{B} \cap A )$ is a standard measurable space. We denote by $\mathcal{B}_{+} = \{A \in \mathcal{B}: m(A)> 0 \}$, the induced measure space is $(A, \mathcal{B} \cap A, m_{|_A} )$ where $m_{|_A}(B) = m(A \cap B)$ and this is standard by the above.

A measurable map $T:X \rightarrow X$ is called measure preserving if $m(C)=m(T^{-1}(C)$ for all $C \in \mathcal{B} $.
\begin{definition}
	A continuous flow
	$\{\phi_t\}_{t \in \mathbb{R}}$ on $X$ preserves a measure $m$  if each $\phi_t$ is a measure preserving transformation with respect to $m$.
\end{definition}
We say that a set $Z \subset X$ is $\phi_t$-invariant if for all $t \in \Bbb{R}$, $\phi_t(Z) =Z \mod m$.
%\begin{definition}
%	A measure preserving transformation T of X is  a measurable, measure preserving map $T:Y \rightarrow Y$  with $Y \in \mathcal{B}$ and $m(X\backslash Y) = 0$.
%\end{definition}

%\begin{thm}[Poincar\'{e} Recurrence Theorem]\label{PRT}
%	If $T:M\to M$ is  a measure-preserving transformation and $m(X) < + \infty$
%	then, $m$-almost all point $p\in M$ is recurrent for $T$.
%\end{thm}
Now we state the Maximal ergodic theorem that will be used in the proof of the main result of this paper. This result can be found in \cite{p:t}.
%\begin{thm}(Birkhoff Ergodic Theorem)
%	Let $T:M\to M$ be a  a measure-preserving transformation with $m(X) < + \infty$. If f is integrable function then the limit,
%	$$   f^{*}(x) = \lim_{n\to+\infty} \displaystyle\frac{1}{n} \displaystyle\sum_{j=0}^{n-1} f(T^j(x))                                               $$
%	exists for m-a.e $x \in X$ and 
	%$$     \displaystyle\int_{X}f dm =    \displaystyle\int_{X}f^{*} dm                     $$
%\end{thm}
¨%\begin{thm}(Birkhoff Ergodic Theorem for flow)
%	Let $\{\phi_t\}_{t \in \mathbb{R}}$ be a a measure-preserving flow. If $m(X) < + \infty$  and $f:M \rightarrow \mathbb{R}$ is a m-integrable function. Then,
%	$$ f_t^{+}(x) = \displaystyle\frac{1}{t} \displaystyle\int_{0}^{t}f(\phi_t(x)) ds \ \ \ \ \ \rm{and} \ \ \ \ \  f_t^{-}(x) =  \displaystyle\frac{1}{t} \displaystyle\int_{0}^{t}f(\phi_{-t}(x)) ds $$
%	converge almost everywhere to the same m-integrable and $\{\phi_t\}$-invariant limit
	%function $f^{*}$, and $\displaystyle\int_{M} f^{*} dm = \displaystyle\int_{M} f dm $.
%\end{thm}
\begin{theorem}(Maximal ergodic theorem)
	Let $M$ be a complete Riemannian manifold equipped with the $\sigma$-algebra $\mathcal{B}$  of Borel sets, a flow  $\{\phi_t\}_{t \in \mathbb{R}}$  and a invariant measure for the flow $\mu$. If $f$ is a measurable function such that $f^{+}$ or $f^{-}$ is integrable and $Z \subset M$ is a $\phi_t$-invariant Borel set then
	$$ \displaystyle\int_{Z_f} f \, d\mu \geq 0,        $$
	where $Z_f = \{ x \in Z| \displaystyle\sup_{s > 0} \displaystyle\int_{0}^{s} f(\phi_t(x)) \, dt >0\}.$
\end{theorem}
Observe that in the hypothesis of the Maximal Ergodic Theorem the Riemannian manifold $M$ does not necessarily has finite volume.

	We say that $p\in X$ 
	is a recurrent point for a continuous flow
	$\{\phi_t\}_{t \in \mathbb{R}}$ on $X$ if there exists a sequence
	$t_j\to+\infty$ in $\mathbb{N}$ such that $\phi_{t_j}(p)\to p$.
\begin{definition}
	We say that a flow $\{\phi_t\}_{t \in \mathbb{R}}$ that preserves the measure $m$ is recurrent if, given any measurable set $A$, for $m$-a.e $x \in A$ there exists a sequence $t_n \rightarrow \infty$ such that $\phi_{t_n}(x) \in A$.
\end{definition}

\subsection{The decomposition of E.Hopf and a measure on the spaces of orbits}
Following \cite{flo:gui}, let $f_0> 0$ be an integrable function on $SM$ and $\{\phi_t\}_{t \in \mathbb{R}}$ the geodesic flow. Then the Borel sets
$$ D^{+} = \Big\{ \theta \in SM: \displaystyle\int_{0}^{\infty} f_0(\phi_t(\theta)) \,  dt < \infty      \Big\},  \ \ \ \  C^{+} = SM \backslash D^{+}                  $$
are $\phi_t$-invariant. Furthermore, they are independent of $f_0$ in the following sense: if $f_1 \geq 0$ is another integrable function, then $\nu_g^{S}(E) = 0$, where
$$    E = \Big\{\theta \in SM: \displaystyle\int_{0}^{\infty}  f_1(\phi_t(\theta)) \,  dt = \infty                   \Big\} \cap D^{+}.                               $$
The decomposition $SM = D^{+} \cup C^{+}$ is called E. Hopf's decomposition of $SM$ associated to the geodesic flow $ \{\phi_t\}_{t \in \mathbb{R}}$. The components $D^{+}$ and $C^{+}$ are called, respectively, the dissipative and the conservative parts of the decomposition. Denote by $SM = D^{-} \cup C^{-}$ the decomposition of E. Hopf of $SM$ associated to the inverse flow $ \{\phi_{-t}\}_{t \in \mathbb{R}}$. Then
$$ D^{-} = \Big\{ \theta \in SM: \displaystyle\int_{-\infty}^{0} f_0(\phi_t(\theta)) \,  dt < \infty      \Big\}.                        $$
From the above discussions, if $f$ is an integrable function on $SM$ then for almost all $\theta \in D = D^{+} \cap D^{-}$, the Lebesgue integral
\begin{equation}\label{in} \displaystyle\int_{-\infty}^{\infty} f(\phi_t(\theta)) \,  dt < \infty,
\end{equation}
exists. 
Let $\sim$ be the equivalence relation in $SM$ defined by: $\theta \sim \eta$ if and only if there exists $j \in \Bbb{Z}$ such that $\phi_j(\theta) = \eta$. Let $SM / \sim$ be the space of orbits, and denote by $\pi:SM \rightarrow SM/\sim$ the natural projection $\pi(\theta) = [\theta]$. Consider in $SM/\sim$ the $\sigma$-algebra $\tilde{\beta}$ induced by $\pi$. Then $\pi(E)$ is measurable for every Borel set $E$. A Borel set $E$ is called a wandering set if $\phi_j(E) \cap E = \emptyset$ for every $j \geq 1$. For each $\tilde{E} \in \tilde{\beta}$, let 
$$    \tilde{\mu}(\tilde{E}) = \sup\{\nu_g^{S}(E): E \subset \pi^{-1}(\tilde{E}), \ \ E \ \ \rm{is \ \ a \ \ wandering \ \ set} \}.                            $$
Guimar\~aes proved in \cite{flo:gui} that $\tilde{\mu}$ is a measure on $SM/\sim$ with the property that $\tilde{\mu}(\pi(E)) = \mu(E)$ for every wandering set. Given a function $f$ integrable on $SM$, then by (\ref{in}), $[\theta] \mapsto \displaystyle\int_{-\infty}^{\infty} f(\phi_t(\theta)) \,  dt$, $[\theta] \in \pi(D)$, defines a $\tilde{\mu}$-measurable function on $\pi(D)$.
\begin{proposition}\label{id}(Proposition 2.3,\,\cite{flo:gui})
If f is an integrable function on $SM$, then
$$      \displaystyle\int_{D} f d\nu_g^{S} = \displaystyle\int_{\pi(D)} \displaystyle\int_{-\infty}^{\infty} f(\phi_t(\theta)) \, dt d\tilde{\mu}.                           $$
\end{proposition}
\section{Proof of Theorem \ref{infinito}}
We denote by $\nu_g^{S}$ the Riemannian measure on $SM$ induced by Sasaki metric. In the next lemma we relate the integral of the function $f_X$ on $SM$ with the integral of the divergence of $X$ on $M$.

\begin{lemma}\label{es}
	Let $(f_X)_p$ be the restriction of $f_X$ to  $\pi^{-1}(p)$. Then
	$$\overline{f_X}(p) = \displaystyle\int_{\pi^{-1}(p)} (f_X)_p \, d\nu_{g_p}^{S}
	% = \displaystyle\int_{\pi^{-1}(p)} (f_p) \,  d\nu_{g_p^S}   	
	= \displaystyle\frac{\omega_{n-1}}{n}\cdot {\rm div}\, X(p) ,   $$
	where $\omega_{n-1}$ is the volume of the $n-1$ dimensional sphere in $\mathbb{R}^n$ with the canonical metric.
\end{lemma}
\begin{proof}
	Fix a orthonormal basis $\{e_1,...,e_n\}$ in $T_pM$  and consider the charts $\varphi_{+}:U_{+} \rightarrow  B$ and $ \varphi_{-}:U_{-} \rightarrow B$, where
	$$B = \{(x_1,...,x_{n-1}) \in \mathbb{R}^{n-1}: x_1^2 + \cdot \cdot \cdot +x_{n-1}^2 < 1\},$$
	$U_{+} = \{w \in \pi^{-1}(p): g(w,e_n) > 0\}$ and 	$U_{-} = \{w \in \pi^{-1}(p): g(w,e_n) < 0\}$. The maps $\varphi_{+}$ and  $\varphi_{-}$ are defined by $$\varphi_{+}^{-1}(x_1,...,x_{n-1}) = \displaystyle\sum_{i=1}^n x_ie_i \ \  {\rm{and}} \ \ 
	\varphi_{-}^{-1}(y_1,...,y_{n-1}) = \displaystyle\sum_{i=1}^n y_i e_i,$$ with $x_n = \sqrt{1-(x_1^2+ \cdot \cdot \cdot + x_{n-1}^2)}$
	and 
	$y_n = -\sqrt{1-(y_1^2+ \cdot \cdot \cdot + y_{n-1}^2)}$.
	Therefore,
	
	\begin{eqnarray*}
		\overline{f_X}(p) = &\displaystyle\int_{B}&  (\sqrt{G^{+}} \circ \varphi_{+}^{-1}(x)) \Big(\displaystyle\sum_{i,j=1}^n x_ix_j g(\nabla_{e_i}X,e_j)\Big) \,  dx_1 \cdot \cdot \cdot dx_{n-1}\\
		&+&  \displaystyle\int_{B}  (\sqrt{G^{-}} \circ \varphi_{-}^{-1}(y)) \Big(\displaystyle\sum_{i,j=1}^n y_iy_j g(\nabla_{e_i}X,e_j)\Big) \, dy_1 \cdot \cdot \cdot dy_{n-1},\\
	\end{eqnarray*}
	where $G^{+}$ denotes the determinant of a matrix $(g_{ij}^{+})$ consisting of the components of g with respect to local coordinates $(x_i)$ and $G^{-}$ denotes the determinant of a matrix $(g_{ij}^{-})$ consisting of the components of g with respect to the local coordinates $(y_i)$.
	On the other hand, if $i \neq j$ and $i,j \leq n-1$ we have,
	$$   \displaystyle\int_{B} x_ix_j (\sqrt{G^{+}} \circ \varphi_{+}^{-1}(x)) \,  dx_1 \cdot \cdot \cdot dx_{n-1}     = \displaystyle\int_{\mathbb{S}_{+}^{n-1}} z_iz_j \, d \mathbb{S}^{n-1}    =0,            $$
	where  $\mathbb{S}_{+}^{n-1} = \{(x_1,...,x_n) \in \mathbb{S}^{n-1}: x_n >0\}$. For $i \leq n-1$ we have,
	$$      \displaystyle\int_{B} x_ix_n (\sqrt{G^{+}} \circ \varphi_{+}^{-1}(x))  \, dx_1 \cdot \cdot \cdot dx_{n-1}     = \displaystyle\int_{\mathbb{S}_{+}^{n-1}} z_i\sqrt{1-(z_1^2+ \cdot \cdot \cdot + z_{n-1}^2)} \, d \mathbb{S}^{n-1}    =0,                      $$
	and
	$$        \displaystyle\int_{B} x_i^2 (\sqrt{G^{+}} \circ \varphi_{+}^{-1}(x)) \,  dx_1 \cdot \cdot \cdot dx_{n-1}     = \displaystyle\int_{\mathbb{S}_{+}^{n-1}} z_i^2 \, d \mathbb{S}^{n-1}   = \displaystyle\frac{\omega_{n-1}}{2n}.        $$
	
	If $i=n$ we have,
	
	\begin{eqnarray*}
		\displaystyle\int_{B} x_n^2 (\sqrt{G^{+}} \circ \varphi_{+}^{-1}(x)) \,  dx_1 \cdot \cdot \cdot dx_{n-1}     &=& \displaystyle\int_{\mathbb{S}_{+}^{n-1}} 1-(z_1^2+ \cdot \cdot \cdot + z_{n-1}^2) \, d \mathbb{S}^{n-1}\\
		&=&   \displaystyle\frac{\omega_{n-1}}{2}- \displaystyle\frac{n-1}{2n}\omega_{n-1}   \\
		&=&     \displaystyle\frac{\omega_{n-1}}{2n}.
	\end{eqnarray*}
	
	Therefore,
	\begin{eqnarray*}
		\displaystyle\int_{B} (\sqrt{G^{+}} \circ \varphi_{+}^{-1}(x)) \Big(\displaystyle\sum_{i,j=1}^n x_ix_j g(\nabla_{e_i}X,e_j)\Big) \, dx_1 \cdot \cdot \cdot dx_{n-1} &=& 
		\displaystyle\frac{\omega_{n-1}}{2n}\displaystyle\sum_{i=1}^n g(\nabla_{e_i}X,e_i) \\
		&=& \displaystyle\frac{\omega_{n-1}}{2n} {\rm div}\, X(p).\\
	\end{eqnarray*}
	Proceeding in the same way, we get
		$$\displaystyle\int_{B}  (\sqrt{G^{-}} \circ \varphi_{-}^{-1}(y)) \Big(\displaystyle\sum_{i,j=1}^n y_iy_j g(\nabla_{e_i}X,e_j)\Big) \, dy_1 \cdot \cdot \cdot dy_{n-1} = \displaystyle\frac{\omega_{n-1}}{2n} {\rm div}\, X(p).$$
	Hence, $$\overline{f_X}(p) = \displaystyle\int_{\pi^{-1}(p)} (f_X)_p \,  d\nu_{g_p^S}   = \displaystyle\frac{\omega_{n-1}}{n}\cdot {\rm div}\, X(p).$$ \ \ \ \  \ \ \ \ \ \ \ \  \ \ \ \  \ \ \ \  \ \ \ \  \ \ \ \  \ \ \ \ \ \ \  \ \ \ \ \ \ \  \ \ \ \ \   \ \ \ \ \ \  \ \ \ \ \  \ \ \ \   \ \ \  \ \ \ \ \ \ \ \ \ \  \ \ \ \ \ \ \ \ \  \ \ \  \  \ \ \ \  \ \  \ \ \ \ \ \ \ \ \ 
\end{proof}

	Fix an integrable function $h$ on $M$ strictly positive such that \linebreak $\displaystyle\int_{0}^{1}h(\gamma_{\eta}(t)) \, dt > 0$ for every $\eta \in SM$. Now consider a recurrent point $\theta=(p,v)$ in $SM$ and the function $\tilde{h}$ on $SM$ where $\tilde{h}(p,v) = h(p)$. We have
\begin{eqnarray*}
	\displaystyle\int_{0}^{t} (f_X + \tilde{h})(\phi_t(\theta)) \, ds &=&\displaystyle\int_{0}^{t}  \displaystyle\frac{d}{ds}g(X(\gamma_{\theta}(s)),\gamma'_{\theta}(s)) \, ds + \displaystyle\int_{0}^{t} h(\gamma_{\theta}(s)) \, dt\\
	&=& g(X(\gamma_{\theta}(t)),\gamma'_{\theta}(t)) - g(X(p),v) + \displaystyle\int_{0}^{t} h(\gamma_{\theta}(s)) \, dt.\\
\end{eqnarray*}
Since $\theta$ is recurrent there exists a sequence $t_n \rightarrow + \infty$ such that $\phi_{t_n}(\theta) \rightarrow \theta$ thus 
$g(X(\gamma'_{\theta}(t_n)),\gamma'_{\theta}(t_n)) \rightarrow g(X(p),v)$. Hence, $\displaystyle\sup_{t > 0} \displaystyle\int_{0}^{t} (f_X + \tilde{h})(\phi_s(x)) \, ds >0$. Since the geodesic flow is recurrent, for $\nu_g^S$-a.e $ \theta \in SM$ the point $\theta$ is recurrent. Applying the Maximal ergodic theorem, we conclude that 
$$       \displaystyle\int_{SM} (f_X + \tilde{h}) \, d\nu_{g^S} \geq 0.                                       $$
Repeating this argument with the function $1/k \cdot h$, it follows that for every $n \in \mathbb{N}$,
$$    \displaystyle\int_{SM} f_X \,  d\nu_{g^S} \geq -\displaystyle\frac{1}{k}    \displaystyle\int_{SM} \tilde{h} \, d\nu_{g^S}                                   $$            
Hence,    $$  \displaystyle\int_{SM} f_X  \, d\nu_{g^S} \geq 0.$$                  
On the other hand, the function $f_{-X}$ is also integrable on $SM$. Applying the argument above, it follows that
$$            \displaystyle\int_{SM} f_X  \, d\nu_{g^S} \leq 0.$$         
Therefore,  
$    \displaystyle\int_{SM} f_X \,d\nu_{g^S} = 0                               $. From Theorem \ref{FU} and Lemma \ref{es}, we have that ${\rm div}\, X$ is integrable and
\begin{eqnarray*}
	\displaystyle\int_{M} {\rm div}\, X \,d\nu_g\  &=&\displaystyle\frac{n}{\omega_{n-1}}\displaystyle\int_{SM} f_X \,d\nu_{g^S}. \\
	&=& 0.
\end{eqnarray*}  
This concludes the proof of Theorem \ref{infinito}. If $(M,g)$ has finite volume then from Poincar\'e recurrence theorem  it follows that the geodesic flow $\phi_t$ is recurrent. As an immediate consequence of Theorem \ref{infinito}, we get the Corollary \ref{finito}. The Maximal ergodic theorem was also used by Guimar\~aes in \cite{flo:gui} to get a rigidity result on Riemannian manifolds without conjugate points. 

\section{Proof of Theorem \ref{infinito2}}
Observe that $SM$ is the disjoint union $SM = D \cup C^{+} \cup (C^{-}\backslash C^{+})$. Since $C^{+}$ is $\phi_t$-invariant and for almost all $\theta \in C^{+}$, $\theta$ is recurrent, by the proof of Theorem \ref{infinito} we have
$$    \displaystyle\int_{C^{+}} f_X  \,d\nu_{g^S} = 0.                           $$

Note that $C^{-}\backslash C^{+}$ is also $\phi_t$-invariant and for almost all $\theta \in C^{-}$, $\theta$ is recurrent with respect to the inverse geodesic flow. By the proof of Theorem \ref{infinito} we have
$$    \displaystyle\int_{(C^{-}\backslash C^{+})} f_X  \,d\nu_{g^S} = 0.                           $$

Let $d,\tilde{d}$ be the distances on $M$ and $SM$ respectively, for almost all $\theta \in D$ we have $\tilde{d}(\theta, \phi_t(\theta)) \rightarrow \infty $  as $|t| \rightarrow \infty$. On the other hand, we have          $\tilde{d}(\theta, \phi_t(\theta)) \leq d(p, \gamma_{\theta}(t))+ 2\pi$, where $\theta = (p,v)$. Thus, for almost $\theta \in SM$ we get $d(p, \gamma_{\theta}(t)) \rightarrow \infty $  as $|t| \rightarrow \infty$.

%Now fix a point $\eta \in SM$, for almost all $\theta \in D$ we have $\tilde{d}(\eta, \phi_t(\theta)) \rightarrow \infty  $ as $|t| \rightarrow \infty$.

 On the other hand,
$$   \Big|\displaystyle\int_{-s}^{s} f_{X}(\phi_t(\theta))  \, dt \Big|\leq |X(\gamma_{\theta}(s))| + |X(\gamma_{\theta}(-s))|.                          $$
Since $f_X$ is integrable and $|X| \rightarrow 0$ at infinity in $M$, for almost all $\theta \in D$ it follows that
$$         \displaystyle\int_{-\infty}^{\infty} f_{X}(\phi_t(\theta))  \, dt   =0.                                $$
By Proposition \ref{id} follows that
$$           \displaystyle\int_{D} f_{X} \, d\nu_g^{S} =0.         $$
Therefore,
\begin{eqnarray*}
 \displaystyle\int_{SM} f_{X} \, d\nu_g^{S} &=& \displaystyle\int_{D} f_{X} \, d\nu_g^{S} + \displaystyle\int_{C^{+}} f_{X} \, d\nu_g^{S} + \displaystyle\int_{{(C^{-}\backslash C^{+})}} f_{X} \, d\nu_g^{S}.\\
 &=& 0.  
\end{eqnarray*}                                From Theorem \ref{FU} and Lemma \ref{es}, we have that ${\rm div}\, X$ is integrable and
\begin{eqnarray*}
\displaystyle\int_{M} {\rm div}\, X  \,d\nu_g\  &=&\displaystyle\frac{n}{\omega_{n-1}}\displaystyle\int_{SM} f_X \,d\nu_{g^S}. \\
&=& 0.
\end{eqnarray*}       
\section{Sufficient conditions for the integrability of $f_X$}
\label{sec-suficiente}

Here we present a proposition that guarantee the integrability of $f_X$ which is used in the examples in the next section.

\begin{proposition}\label{suficiente}
	Let $M$ be a complete non-compact Riemannian manifold without boundary and X a $C^1$ vector field such that
	\begin{eqnarray*}
	\displaystyle\liminf_{r\rightarrow +\infty} \displaystyle\frac{1}{r} \displaystyle\int_{  B(2r)/B(r) }       |X| \, d\nu_g  &<& \infty.\\                
	\end{eqnarray*} 
	If $f_X$ has an integral (i.e either $f_X^{+}$ or  $f_X^{-}$ is integrable) then $f_X$ is integrable.
\end{proposition}

	Without loss of generality, we may assume that $f_X^-$ is integrable. It is known (see  \cite{yau:yau}) that there is a constant $C > 0$ such that for each $r> 0$ exists a continuous function $\varphi_r$ satisfying: $ 0 \leq \varphi_r \leq 1$, $\varphi_r \equiv 1$ on $B(r)$, $ \varphi_r \equiv 0$ on the complement of $B(2r)$  and $|| {\rm grad}\, \varphi_r(x)|| \leq C /r$. Consider the function defined by $f_r(p,v) = \varphi_r(p) f_X(p,v)$. Since $f_r$ has a compact support it follows from Proposition \ref{es} and Theorem \ref{FU} that
$$            \displaystyle\int_{SM} f_r \, d\nu_{g^S} = %\displaystyle\lim_{n\to+\infty}	\displaystyle\int_{SM} f_nd\nu_{g^S}\\
%&=& \displaystyle\lim_{n\to+\infty}\displaystyle\frac{\omega_{n-1}}{n}\displaystyle\int_{M} \psi_{n}(p) (\rm{div}X(p)) d\nu_g\\
\displaystyle\frac{\omega_{n-1}}{n}\displaystyle\int_{M}\varphi_r \cdot {\rm div}\, X  \, d\nu_g .               $$ 

On the other hand, we have that $ {\rm div}\, (\varphi_r X) = \varphi_r \cdot {\rm div}\, X + g({\rm grad}\, \varphi_r, X )                   $. From the classic divergence theorem it follows that
$$        \displaystyle\int_{M}\varphi_r \cdot {\rm div}\, X  \, d\nu_g       = -    \displaystyle\int_{M}   g(\rm{grad}\, \varphi_r, X )  \,  d\nu_g.      $$ 
Hence,
$$\Big|      \displaystyle\int_{M}(\varphi_r \cdot {\rm div}\, X) \, d\nu_g                        \Big|      \leq \displaystyle\frac{C}{r}    \displaystyle\int_{  B(2r)/B(r) }       |X| \, d\nu_g.              $$  
Consider the set defined by $A_r = \{(p,v): p \in B(2r)\}$. We have 
\begin{eqnarray*}
	\displaystyle\int_{A_r} f_r^{+} \,  d\nu_{g^S} - \displaystyle\int_{SM} f_X^{-} \, d\nu_{g^S} &\leq& \Big|  \displaystyle\int_{SM} f_r  \, d\nu_{g^S}   \Big|  \\
	&\leq& \displaystyle\frac{C \cdot \omega_{n-1} }{nr}       \displaystyle\int_{  B(2r)/B(r) }       |X| \,  d\nu_g           \\
\end{eqnarray*}   
We may choose $r_i \rightarrow +\infty$ such that the right-hand side of this inequality is bounded. Thus $f_X^{+}$ is also  integrable and therefore $f_X$ is integrable.

	\section{Examples}\label{example}

In this section  we present four examples related to  the hypothesis we assume and we discuss their relation with the existing results; more precisely, we provide different type of examples satisfying our hypothesis which are not covered by previous results. 

In the first example we construct a smooth vector field $X$ that satisfies the hypothesis of Theorem \ref{teokarp}, $f_X$ is integrable on $SM$ but $|X|$ is non-integrable on $M$.

In the second example we construct a complete non-compact Riemannian manifold with finite volume $M_{\alpha}^{3}$ and a smooth vector field $\overline{Z}$ that does not satisfy the hypothesis of Theorem \ref{teokarp},  but  $f_X$ is integrable on $SM$. Furthermore, $|\overline{Z}|^p$ is not integrable for every $p \geq 1$. This example shows that the condition of integrability of the function $f_X$ is really different of the sufficient conditions used in \cite{dihe:newdir}, \cite{l:karp}, \cite{go:tro} and \cite{val:ve} to ensure that the integral of the divergence is null. 

Furthermore, in the third example  we construct a complete Riemannian manifold with infinite volume $M$ and a smooth vector field $\overline{U}$ that satisfies the hypothesis of Theorem \ref{infinito2} but  does not satisfy the sufficient conditions used in \cite{dihe:newdir}, \cite{l:karp}, \cite{go:tro} and \cite{val:ve} to ensure that the integral of the divergence is null.

When the Riemannian manifold is compact the integral of the divergence of a $C^1$ vector field $X$ is zero. An interesting question is to know whether this phenomenon also occurs
in non-compact Riemannian manifold with finite volume for the vector fields whose divergence is integrable. In the last example we will answer this question by constructing an example of a complete Riemannian manifold non-compact with finite volume $N$ and a smooth vector field $Z$ on $M$ such that ${\rm div}\, Z$ is integrable on $M$, but 
$$  \displaystyle\int_{M}  {\rm div}\, X \, d\nu_{g} > 0.$$

	\begin{example}
		Let $N\subset\mathbb{R}^3$ be the surface obtained by rotating the graph of the function 
		$f(x)= \displaystyle\frac{1}{1+ x^2}$, where $x\in\mathbb{R}$, around the axis $x$ with the usual metric. Then, $S$ is a complete 
		non-compact surface. 
		
		We will show below that $S$ has finite area. We have
		\[
		\nu_g(N)=2\pi\int_{-\infty}^{+\infty}\displaystyle\frac{1}{1+x^2}\sqrt{1+\displaystyle\frac{4x^2}{(1+x^2)^4}} \, dx
		=4\pi\int_{0}^{+\infty}\displaystyle\frac{1}{1+x^2}\sqrt{1+\displaystyle\frac{4x^2}{(1+x^2)^4}}  \, dx.
		\]   	
		Since $\displaystyle\lim_{x\to+\infty}   \displaystyle\frac{4x^2}{(1+x^2)^4}    =0$ there exists $x_0\in
		\mathbb{R}$ such that $ \displaystyle\frac{4x^2}{(1+x^2)^4}  \leq3$, for all $x\geq x_0$. Therefore, 
		\[
		\nu_g(N)\leq4\pi\int_{0}^{x_0}    \displaystyle\frac{1}{1+x^2}\sqrt{1+\displaystyle\frac{4x^2}{(1+x^2)^4}}  \,   dx
		+8\pi\int_{x_0}^{+\infty} \displaystyle\frac{1}{1+x^2} \, dx <    +\infty.
		\]
		
		We can also write $N = G^{-1}(\{0\})$ where $G: \mathbb{R}^3 \rightarrow \mathbb{R}$ is given by 
		$$G(x,y,z) = y^2+z^2-\displaystyle\frac{1}{(1+x^2)^2}.$$
		Hence, if $p=(x,y,z) \in M$ then  
		$$T_pN = \Big\{(a,b,c):g\Big((a,b,c),\Big(\displaystyle\frac{4x}{(1+x^2)^3},2y,2z\Big)\Big) = 0  \Big\}.$$
		
		Consider the smooth vector field $W(x,y,z)= x(1+x^2) \cdot (0,-z,y) $.  The function $|W|$ is non-integrable because
		$$      \displaystyle\int_{N}|W| \,  d\nu_g         = \displaystyle\lim_{r\to+\infty} \displaystyle\int_0^{2\pi} \displaystyle\int_{-r}^{r} \displaystyle\frac{|s|}{1+s^2}\sqrt{1+\displaystyle\frac{4s^2}{(1+s^2)^4}} \, ds dt  = + \infty.          $$ 
		
		On the other hand, given $v=(a,b,c) \in T_pN$ with 
		$a^2+b^2+c^2=1$ follows that
		\begin{eqnarray*}
			f_{W}(p,v) &=& g(\nabla_v W,v)\\
			&=& a(1+3x^2)(-zb+cy).\\
		\end{eqnarray*}
		Hence, ${\rm div}\, W = 0$ and
		\begin{eqnarray*}
			|f_{W}(p,v)| &\leq& (1+3x^2)\sqrt{b^2+c^2} \sqrt{z^2+y^2}\\
			& \leq& \displaystyle\frac{1+3x^2}{1+x^2}\\
			& \leq & 3.\\
		\end{eqnarray*}
		Therefore $f_W$ is integrable on $SN$.

		Now consider the points $p=(0,1,0) $, $q=(x,y,z)$ and $q_1 = \Big(x,\displaystyle\frac{1}{1+x^2},0\Big)$ in $M$. Let $d$ be the distance on $M$. We have 
		$$ \sqrt{x^2+ (y-1)^2+z^2}                \leq   d(p,q) \leq  d(p,q_1) + d(q_1,q).$$
		
		%We have 
		
		%	$$      d(p,q_1) = \displaystyle\int_{0}^{|x|}   \sqrt{1+\displaystyle\frac{4s^2}{(1+s^2)^4}} ds\geq |x|                      $$

		Observe that there exists $C > 1$ such  that
		$$   |x|    \leq  d(p,q) \leq C|x| + \displaystyle\frac{\pi}{1+x^2} \leq C|x| + \pi.$$
		%	Consider $r_0 > 1$ such that if $r > r_0 $ then $c_r < d_r$ where $c_r = \sqrt{r^2-4}$, $d_r= \pi +2r$ and $\sqrt{r_0^2-4} > x_1$. 
		
		For each $r$ consider the set $ F_r$ defined by, 
		$$F_r=\{(x,y,z) \in M: r <|x| \leq 2Cr + 2\pi\}.$$

		From the above inequalities it follows that $  B(2Cr + 2\pi) /B(Cr+\pi) \subset F_r$. Hence, for $r$ large we get
		\begin{eqnarray*}
			\displaystyle\int_{  B(2Cr + 2\pi)/B(Cr +\pi) }       |W| d\nu_g &\leq& \displaystyle\int_{F_r} |W| \, d\nu_g \\
			&=& 4\pi \displaystyle\int_{r   }^{2Cr + 2\pi} \displaystyle\frac{|s|}{1+s^2}\sqrt{1+\displaystyle\frac{4s^2}{(1+s^2)^4}} \, ds\\
			& \leq& 8\pi \displaystyle\int_{r   }^{2Cr+2\pi} \displaystyle\frac{|s|}{1+s^2} \, ds\\
			&=& 4\pi(\log (2Cr + 2\pi) -  \log r).\\
		\end{eqnarray*}
		Therefore,
		
		$$    \displaystyle\liminf_{r\rightarrow +\infty} \displaystyle\frac{1}{r} \displaystyle\int_{  B(2r)/B(r) }       |W| \, d\nu_g    \leq  \displaystyle\lim_{r\rightarrow +\infty} \displaystyle\frac{4\pi(\log(2Cr+2\pi)- \log r)}{r}    = 0  .     $$

	\end{example}
	%then for $r >2$, if $q \in B(2r) /B(r)$, $a_r = \sqrt{r^2-4}$ and $b_r = \sqrt{4r^2 +\pi} $ then 
	%$$a_r     \leq |x| \leq b_r $$
	%Thus $B(2r) /B(r) \subset A$, where $A = \{(x,y,z)\in M:a_r = \sqrt{r^2-4}        \leq |x| \leq \sqrt{4r^2 +\pi} \}$.
	%If $d(p,q) \geq r$ then $\sqrt{x^2+4} \geq\sqrt{x^2 +2 + 2e^{-x^2}} \geq r$, therefore $x^2 \geq r^2 -4$. If $d(p,q) \leq r$ then $  \sqrt{x^2+ (1-e^{-x^2})^2}       -\pi e^{-x^2} \leq r$, therefore $x^2 \leq r^2 + \pi$.

	%Therefore,
	%\begin{eqnarray*}
	%\displaystyle\int_{  B(2r)/B(r) }       |Z| d\nu_g &\leq& \displaystyle\int_{A} |Z| d\nu_g \\
	%&\leq& 4\pi \displaystyle\int_{a_r   }^{b_r} 
	%\end{eqnarray*}
	
%\end{example}
\begin{example}
	Let $(M,g)$ be the warped product  $\mathbb{H}^2 \times_{h} \mathbb{S}^1$ where $\mathbb{H}^2$ is the \linebreak hyperbolic plane and $\mathbb{S}^1$ is the unit circle. Fix a point $(y,a) \in M$ and \linebreak consider the real function $h: \mathbb{H}^2 \rightarrow (0, \infty)$ defined by $h(p) = b(d(p,y))$, where $d$ denotes the distance in $\mathbb{H}^2$ and $b: \mathbb{R} \rightarrow (0,\infty)$ is a $C^{\infty}$ function such that $b(r) = a > 0$ if $|r| < 1$, $\displaystyle\lim_{r\rightarrow +\infty} b(r) = 0$, $\displaystyle\int_{0}^{\infty} b(r) \cdot \sinh r \, dr < \infty$ and $\displaystyle\lim_{r\rightarrow +\infty} (\sinh r)^2 \cdot b(r)  > 0$. 
	
	 Now consider the parametrization 
	$\Psi: (0, \infty) \times (0,2\pi) \times (0,2\pi) \rightarrow M$, given by $\Psi(r, \theta, t) = ( \varphi(r,\theta), \eta (t))$, where $ \varphi(r,\theta) = \exp_y (r \gamma(\theta))  $, $\gamma(\theta)$ is a circle of radius 1 in $T_y \mathbb{H}^2$ with center 0 parametrized by the central angle $\theta$ and $\eta(t)$ is the parametrization of $\mathbb{S}^1$ by the central angle. We have,
			$$
		\nu_g(M) = 4\pi^2\displaystyle\int_{0}^{\infty}\displaystyle \sinh r \cdot b(r)  \, dr.
$$
	Therefore $(M,g)$ is a complete non-compact Riemannian manifold  with finite volume. Now consider the unique smooth vector field $Z$ on $\mathbb{H}^2$ such that $Z(   \exp_y (r \gamma(\theta))  = \displaystyle\frac{\partial  \varphi}{\partial \theta} (r,\theta)     $ for every $(r,\theta) \in (0, \infty) \times (0,2\pi)$. From Koszul formula it follows that $Z$ is a Killing vector field.
	
	On the other hand, from Proposition \ref{conexao} it follows that the horizontal lift of $Z$ to $M$ is also a Killing vector field. Observe that,
	 $$|| \overline{Z}(  \Psi(r, \theta, t) ) || = \Big|\displaystyle\frac{\varphi}{\partial \theta} (r,\theta)\Big| = \sinh r,$$
	 and
	 $$   \displaystyle\int_{M}|\overline{Z}|^{p} d\nu_g = 4\pi^2\displaystyle\int_{0}^{\infty}\displaystyle (\sinh r)^{p+1} \cdot (b(r)) \, dr.                          $$
	Thus  $|\overline{Z}|^{p}$ is not integrable for every $p \geq 1$.
	Let $\overline{d}$ be the distance on $M$. Given $(p,b) \in M$ notice that
	$$              d(y,p) \leq \overline{d}(  (y,a),(p,b))  \leq    \overline{d}(  (y,a),(p,a))     + \overline{d}(  (p,a),(p,b))    \leq d(y,p) + 2\pi b(   d(p,y)).                  $$
	
	Let $C = \displaystyle\sup_{r \geq 0} b(r)$ and for each $R > \max\{4\pi C,3\} $, let $ A_R$ be the set defined by 
	$$A_R=\{(p,b) \in   M_{\alpha}^3 :  R < d(y,p) < \displaystyle\frac{3}{2}R\}.$$It follows from the above inequalities that  $A_R \subset B(3/2R)/B(R)  $, where $B(R)$ denotes the geodesic ball of radius $R$ and center $(y,a)$. Hence,
	\begin{eqnarray*}
		\displaystyle\int_{  B(2R)/B(R) }       |\overline{Z}| \, d\nu_g &\geq& \displaystyle\int_{A_R} |\overline{Z}| \, d\nu_g \\
		&=& 4\pi^2 \displaystyle\int_{R   }^{3R/2} (\sinh r)^2 \cdot b(r) \, dr.\\
	\end{eqnarray*}
	Which implies that
	
%	$$    \displaystyle\lim_{R\rightarrow +\infty} \displaystyle\frac{1}{R} \displaystyle\int_{  B(2R)/B(R) }       |\overline{Z}| d\nu_g    \geq  \displaystyle\lim_{R\rightarrow +\infty} \displaystyle\frac{4\pi^2}{R}  \displaystyle\int_{R   }^{3R/2} \displaystyle\frac{\sinh r}{r ^{\alpha}} dr.   $$
%	Therefore,  
	$$  \displaystyle\liminf_{R\rightarrow +\infty} \displaystyle\frac{1}{R} \displaystyle\int_{  B(2R)/B(R) }       |\overline{Z}| \, d\nu_g  > 0.           $$
\end{example}
\begin{example}
Let $(M,g)$ be the warped product  $\mathbb{H}^2 \times_{h} \mathbb{S}^1$ where $\mathbb{H}^2$ is the hyperbolic plane and $\mathbb{S}^1$ is the unit circle. Fix a point $(y,a) \in M$ and consider the real function $b: \mathbb{H}^2 \rightarrow (0, \infty)$ defined by $h= b(d(p,y))$, where $d$ denotes the distance in $\mathbb{H}^2$ and $b: \mathbb{R} \rightarrow (0,\infty)$ is a $C^{\infty}$ function such that $b(r) = a > 0$ if $|r| < 1$, $\displaystyle\lim_{r\rightarrow +\infty} b(r) = 0            $ and $\displaystyle\lim_{r\rightarrow +\infty} \sinh r \cdot (b(r))^{p}  > 0$, for every $p \geq 1$. Now consider a smooth vector field $U$ on $\mathbb{S}^1$ such that $ \langle U, U \rangle = 1$ with respect to usual metric of  $\mathbb{S}^1$. Following the same notation of the previous example, by definition we have that  $|\overline{U}(  \Psi(r, \theta, t) )| = b(r)$, where $\overline{U}$ denotes the vertical lift of $U$ to $M$. From Proposition \ref{conexao} it follows that $\overline{U}$ is a Killing vector field on $M$ thus $f_{\overline{U}}$ is integrable on $SM$. This shows that the vector field $\overline{U}$ satisfies the hypothesis of Theorem \ref{infinito2}. 

On the other hand, proceeding in the same way of the previous example we have
$$ \nu_g(M) = 4\pi^2\displaystyle\int_{0}^{\infty}\displaystyle \sinh r \cdot b(r)  \, dr,$$
and

$$  \displaystyle\int_{M}|\overline{U}|^{p} d\nu_g = 4\pi^2\displaystyle\int_{0}^{\infty}\displaystyle \sinh r \cdot (b(r))^{p+1}  \, dr.                             $$
Since $\displaystyle\lim_{r\rightarrow +\infty} \sinh r  \cdot (b(r) )^{q}  > 0$ for every $q \geq 1$ it follows that the Riemannian manifold $M$ has infinite volume and that $|\overline{U}|^{q}$ is not integrable for every $q \geq 1$. By the discussion of the previous example we have

	$$    \displaystyle\lim_{R\rightarrow +\infty} \displaystyle\frac{1}{R} \displaystyle\int_{  B(2R)/B(R) }       |\overline{U}| \, d\nu_g    \geq  \displaystyle\lim_{R\rightarrow +\infty} \displaystyle\frac{4\pi^2}{R}  \displaystyle\int_{R   }^{3R/2} \sinh r \cdot (b(r))^{2} \, dr.   $$
	Therefore,  
	$$  \displaystyle\liminf_{R\rightarrow +\infty} \displaystyle\frac{1}{R} \displaystyle\int_{  B(2R)/B(R) }       |\overline{U}| \, d\nu_g  > 0.           $$

\end{example}
\begin{example}
	Let $\mathbb{L}^{3}$ be the $3$-dimensional Lorentz space, that is, the real
	vector space $\mathbb{R}^{3}$ endowed with the Lorentzian metric
	\[
	\langle p,q\rangle=p_1q_1+p_2q_2-p_3q_3.
	\]
	The $2$-dimensional hyperbolic space
	\[
	\mathbb{H}^{2}=\{p\in\mathbb{L}^{3};\,\langle p,p\rangle=-1,\,p_{3}\geq1\},
	\]
	as it is well known, is a space like hypersurface in $\mathbb{L}^{3}$, that is, the
	induced metric via the inclusion $\iota:\mathbb{H}^{2}\hookrightarrow\mathbb{L}^{3}$ is a
	Riemannian metric on $\mathbb{H}^{2}$. Consider the hyperbolic space $\mathbb{H}^2$ 
	with the orientation given by the normal vector field $N(x,y,z)=(-x,-y,-z)$.
	Let $X\in\mathfrak{X}(\mathbb{H}^2)$ be the vector field defined as the projection 
	of the vector $-e_3$ on each tangent plane of $\mathbb{H}^2$, i.e.,
	\begin{eqnarray*}
		X(x,y,z)=-e_3+\langle -e_3,N(x,y,z)\rangle N(x,y,z)=(xz,yz,z^2-1).
	\end{eqnarray*}
	It is not difficult to see that $X$ is a conformal vector field with conformal factor
	$\psi(x,y,z)=z>0$, that is, $\langle \nabla_VX,U\rangle+\langle\nabla_UX,V\rangle     =    2\psi\langle V,U\rangle   $, where U and V are vector fields on M and $\nabla$ is the Levi-Civita connection on $\mathbb{H}^{2}$.
	
	Let $(M,g)$ be the warped product $\mathbb{H}^2 \times_f \mathbb{S}^1$, where the function \linebreak $f:\mathbb{H}^2 \rightarrow \mathbb{R}$ is given by $f(x,y,z) = 1/z^2$. Now consider the parametrization $\Psi: \mathbb{R}^2\times (0,2\pi) \rightarrow M$, given by $\Psi(x, y, t) = ( \varphi(x,y), \eta (t))$, where $ \varphi(x,y) = (x, y, \sqrt{1+x^2+y^2})  $ and  $\eta(t)$ is the parametrization of $\mathbb{S}^1$ by the central angle. Observe that,
	\begin{eqnarray*}
		\nu_g(M) &=& 2\pi \displaystyle\int_{\mathbb{R}^2} \displaystyle\frac{1}{(1+x^2+y^2)^{3/2}} \, dxdy\\
		&=& 4\pi^2 \displaystyle\int_{0}^{+\infty} \displaystyle\frac{r}{(1+r^2)^{3/2}} \, dr\\
		&=& 4\pi^2.\\
	\end{eqnarray*}
	Therefore $M$ is a complete Riemannian manifold, non-compact with finite volume. Let $Z$ be the horizontal lift of $X$ to $M$. From Proposition \ref{conexao}, it follows that
	\begin{eqnarray*}
		{\rm div}\, Z &=& {\rm div}\, X + Xf/f. 
	\end{eqnarray*}     
	On the other hand, we can write $X = xz \displaystyle\frac{\partial \varphi}{\partial x} + yz  \displaystyle\frac{\partial \varphi}{\partial y}$. Hence  
	$$   Xf = -\displaystyle\frac{2z(x^2+y^2)}{(1+x^2+y^2)^2}.          $$           
	Therefore,
	\begin{eqnarray*}
		{\rm div}\, Z = \displaystyle\frac{2}{\sqrt{    1+x^2+y^2    }}.
	\end{eqnarray*}  
	The $ 	{\rm div}\, Z$ is integrable on $M$. In fact, we have
	\begin{eqnarray*}
		\displaystyle\int_{M} {\rm div}\, Z d\nu_g &=&  2\pi \displaystyle\int_{\mathbb{R}^2}  \displaystyle\frac{2}{(1+x^2+y^2)^2} \, dxdy \\
		&=&  4\pi^2 \displaystyle\int_{0}^{+\infty}  \displaystyle\frac{2r}{(1+r^2)^2} \, dr\\
		&=& 4\pi^2.
	\end{eqnarray*}
\end{example}

\section{Applications}
In this subsection we present some applications related to potential theory. According to the terminology introduced by Rigoli-Setti in \cite{ri:seti}. The $\varphi$-Laplacian of a function $u$ is the nonlinear, divergence form operator defined by
$$  L_{\varphi}(u) = {\rm div}\, \big(|{\rm grad}\, u|^{-1}\varphi(|{\rm grad}\, u|)    \cdot {\rm grad}\, u \big)                     $$
where $\varphi \in C^{0}([0, + \infty)) \cap C^{1}((0, + \infty)) $ satisfies the following structural conditions:
\begin{enumerate}
	\item $\varphi(0) = 0$,
	\item $\varphi(t)> 0$  $\forall t > 0$,
	\item $\varphi(t) \leq At^{r-1}$, for some constants $A,r > 1$.
\end{enumerate}
If $\varphi(t) = t^{p-1}$, $1 < p < +\infty$, the $\varphi$-Laplacian corresponds  to the usual $p-$Laplace operator
$$ \Delta_p u = {\rm       div}\, \big( |{\rm grad}\, u|^{p-2}  \cdot {\rm grad}\, u   \big)                   .$$
In particular, if $p=2$, the $p-$Laplace operator is nothing but the Laplace-Beltrami operator. When $\varphi(t) = \displaystyle\frac{t}{\sqrt{1+t^2}}$, the $\varphi$-Laplacian corresponds to the mean curvature operator 
$$     H_u = {\rm       div}\, \Big(   \displaystyle\frac{{\rm grad}\, u}{\sqrt{1+ {\rm grad}\, u|^{2} }}      \Big). $$ For more details see \cite{pi:setti}. In the theorem below we demonstrate a global comparison principle for complete Riemannian manifolds with infinite volume whose geodesic flow is recurrent. Similar results can be found in \cite{val:ve}, \cite{pi:setti} and \cite{holo:pigola}.

In \cite{holo:pigola} the authors proved a global comparison result for the $p$-Laplacian on a $p$-parabolic manifold and in \cite{pi:setti} we also can find results on a $p$-parabolic manifold. In particular, for closed manifolds and complete Riemannian manifolds with finite volume. It is worth mentioning that the next theorem can be applied to some hyperbolic surfaces with infinite volume and these manifolds are not $p$-parabolic.

\begin{lemma}\label{lemaincri}(Lemma 1.14,\,\cite{pi:setti})
	If $\varphi$ is strictly increasing	and satisfies the  above assumptions then the following holds. Let $(V, \langle, \rangle)$ be an n-dimensional, real vector space endowed with the scalar product $\langle, \rangle$. Then, for every $\xi,\eta \in V$,
	$$      h(\xi,\eta) = \big\langle |\xi|^{-1}\varphi(|\xi|)\xi - |\eta|^{-1}\varphi(|\eta|)\eta, \xi-\eta \big\rangle  \geq0,                                                $$
	with equality holding if and only if $\xi=\eta$.
\end{lemma}

\begin{theorem}
	Assume that $(M,g)$ is a complete Riemannian manifold with infinite volume whose geodesic flow is recurrent with respect to Liouville measure. Consider the $\varphi$-Laplace operator $L_{\varphi}$, where $\varphi$ is strictly increasing and satisfies the  above assumptions. If $u,v \in C^2(M)$ are solutions of $   L_{\varphi}(u) \geq   L_{\varphi}(v) $, $f_X,f_Y$ are integrable on SM where $X= |{\rm grad}\, u|^{-1}\varphi(|{\rm grad}\, u|)    \cdot {\rm grad}\, u $, $Y = |{\rm grad}\, v|^{-1}\varphi(|{\rm grad}\, v|)    \cdot {\rm grad}\, v $  and  $|{\rm grad}\, u|$,$|{\rm grad}\, v| \in L^{r}(M)$  then $u - v \equiv const.$
\end{theorem}
\begin{proof}
	Fix any $x_0 \in M$. Let $A = u(x_0) - v(x_0)$ and define $\Omega_A$ to be the connected component of the open set $\{x \in M: A-1 < u(x) - v(x) < A+1\}$ which contains $x_0$. Let $\alpha:\mathbb{R} \rightarrow [0, + \infty)$ be a smooth function that satisfies the following conditions:
	\begin{enumerate}
		\item $\alpha \equiv 0$ if $t \leq A-1$,
		\item $\alpha \equiv 1$ if $t \geq A+1$,
		\item $\alpha' > 0$ if $A-1 < t < A+1$,
	\end{enumerate}	
	Now consider, the vector field $Z = \alpha \circ (u-v) \cdot (X-Y)$. We have
	$$  f_{Z}(p,w) = \alpha'\circ (u-v) g( {\rm grad}\, u  -{\rm grad}\, v,w)  g(X-Y,w) + \alpha \circ (u-v)(f_X(p,w)-f_Y(p,w)  ).                      $$
	Observe that 
	\begin{eqnarray*}|f_{Z}| &\leq& C.|   {\rm grad}\, u  -{\rm grad}\, v||X-Y| + |f_X|+|f_Y|\\
		&\leq& C(.|   {\rm grad}\, u | +|   {\rm grad}\, v| )   (|X|+|Y|) + |f_X| +|f_Y|\\
		&\leq& C\cdot A( |{\rm grad}\, u |^r +  |{\rm grad}\, v |^r + |{\rm grad}\, v ||{\rm grad}\, u |^{r-1} +  |{\rm grad}\, u ||{\rm grad}\, v |^{r-1} )  \\
		&+& |f_X| +|f_Y|.\\                                 
	\end{eqnarray*}
	By Young's inequality we get
	$$  |{\rm grad}\, u ||{\rm grad}\, v |^{r-1} \leq \displaystyle\frac{ |{\rm grad}\, u |^{r }}{r} +  \displaystyle\frac{ (r-1)|{\rm grad}\, v |^{r }}{r}                                        $$
	and 
	$$  |{\rm grad}\, v ||{\rm grad}\, u |^{r-1} \leq \displaystyle\frac{ |{\rm grad}\, v |^{r }}{r} +  \displaystyle\frac{ (r-1)|{\rm grad}\, u |^{r }}{r}.                                        $$Therefore $f_{Z}$ is integrable, by Theorem \ref{infinito} it follows that ${\rm div}\, Z$ is integrable on $M$ and
	
	$$  \displaystyle\int_{M}  {\rm div}\, Z \,  d\nu_{g}    =0.$$
	On the other hand,
	$$  {\rm div}\, Z= \alpha'\circ (u-v) h( {\rm grad}\, u   ,  {\rm grad}\, v     )
	+ \alpha \circ(u-v)(  L_{\varphi}(u) -   L_{\varphi}(v)    ).$$ 
	By hypothesis, the second term is $\geq 0$. Thus by Lemma \ref{lemaincri}
	$$ 0 =    \displaystyle\int_{M}  {\rm div}\, Z d\nu_{g}  \geq   \displaystyle\int_{M} \alpha'\circ (u-v) h( {\rm grad}\, u   ,  {\rm grad}\, v     ) \, d\nu_{g} \geq 0.                        $$
	Since $\alpha'\circ (u-v) > 0$ on $\Omega_A$ it follows that $h( {\rm grad}\, u   ,  {\rm grad}\, v     ) =0$ on $\Omega_A$. Thus $u-v \equiv A$, on $\Omega_A$. Observe that the open set $\Omega_A$ is also closed. Since $M$ is connected which implies that $\Omega_A = M$ and then $u-v = A$ on $M$.
\end{proof}
An immediate consequence of this theorem is the following corollary.
\begin{corollary}
	Assume that $(M,g)$ is a complete Riemannian manifold with infinite volume whose geodesic flow is recurrent with respect to the Liouville measure. If $u \in C^2(M)$ is a subharmonic function on
	$M$ such that $|{\rm grad}\, u| \in L^2(M) $ and $f_X$ is integrable on $SM$ where $X = {\rm grad}\, u$ then $u$ is constant.
\end{corollary}	
\bibliography{divergencetheorem}

\providecommand{\MR}{\relax\ifhmode\unskip\space\fi MR }
\providecommand{\MRhref}[2]{%
  \href{http://www.ams.org/mathscinet-getitem?mr=#1}{#2}
}
\providecommand{\href}[2]{#2}
\begin{thebibliography}{10}

\bibitem{a:d}
J.~Aaronson and M.~Denker, \textsl{The poincar\'e series of $\mathbb{C}
  \backslash \mathbb{Z}$}, Ergodic Theory Dynam. Systems \textbf{19} (1999),
  no.~1, 1--20.

\bibitem{a:s}
J.~Aaronson and D.~Sullivan, \textsl{Rational ergodicity of geodesic flows},
  Ergod. Th. \& Dynam. Sys. \textbf{4} (1984), no.~2, 165--178.

\bibitem{dihe:newdir}
M.~Gaffney, \textsl{A special stoke's theorem for complete riemannian
  manifolds}, Ann. of Math \textbf{60} (1954), 140--145.

\bibitem{go:tro}
V.~Gol'dshtein and M.~Troyanov, \textsl{The kelvin-nevanlinna-royden criterion
  for p-parabolicity}, Math. Z. \textbf{232} (1999), 607--619.

\bibitem{flo:gui}
F.~F. Guimar\~aes, \textsl{The integral of the scalar curvature of complete
  manifolds without conjugate points}, Journal of Differential Geometry
  \textbf{36} (1992), 651--662.

\bibitem{holo:pigola}
I.~Holopainen, S.~Pigola and G.~Veronelli, \textsl{Global comparison principles
  for the p-laplace operator on riemannian manifolds}, Potential Anal.
  \textbf{34} (2011), 371--384.

\bibitem{hhw:HHW}
Pat Hooper, P.~Hubert and Barak Weiss, \textsl{Dynamics on the infinite
  staircase surface}, Dis. Cont. Dyn. Sys. \textbf{30} (2013), 4341--4347.

\bibitem{l:karp}
L.~Karp, \textsl{On stoke's theorem for noncompact manifolds}, Proc. Amer.
  Math. Soc \textbf{82} (1981), 487--490.

\bibitem{ly:sul}
T.~Lyons and D.~Sullivan, \textsl{Function theory, random paths and covering
  spaces}, J. Differential Geom. \textbf{31} (1984), 299--323.

\bibitem{p:nichols}
P.~Nicholls, \textsl{Transitivity properties of fuchsian groups}, Canad. J.
  Math. \textbf{28} (1976), 805--814.

\bibitem{o:neil}
B.~O'Neill, \textsl{Semi-riemannian geometry with applications to relativity},
  Academic Press, Londres, 1983.

\bibitem{gabriel:paternain}
G.~P. Paternain, \textsl{Geodesic flows}, Progress in Mathematics,
  Birkh$\ddot{\rm a}$user, 1999.

\bibitem{p:t}
K.~Petersen, \textsl{Ergodic theory}, Cambridge University Press, Cambridge,
  1983.

\bibitem{pi:setti}
S.~Pigola and A.~G. Setti, \textsl{Global divergence theorems in nonlinear pdes
  and geometry}, Ensaios Matem\'aticos [Mathematical Surveys], Vol. 26,
  Sociedade Brasileira de Matem\'atica, Rio de Janeiro, 2014.

\bibitem{ri:seti}
M.~Rigolli and A.G. Setti, \textsl{Liouville type theorems for
  $\phi$-subharmonic functions}, Rev. Mat. Iberoamericana \textbf{17} (2001),
  no.~3, 471--520.

\bibitem{sakay:sakay}
T.~Sakay, \textsl{Riemannian geometry}, Translations of Mathematical
  Monographs, \textbf{149} American Mathematical Society, 1996.

\bibitem{tro:tro}
M.~Troyanov, \textsl{Parabolicity of manifolds}, Siberian Adv. Math. \textbf{9}
  (1999), no.~4, 125--150.

\bibitem{val:ve}
D.~Valtorta and G.Veronelli, \textsl{Stoke's theorem, volume growth and
  parabolicity}, Tohoku Math. J. \textbf{63} (2011), 397--412.

\bibitem{yau:yau}
S.~T. Yau, \textsl{Some function-theoretic properties of complete riemannian
  manifolds and their applications to geometry}, Indiana Univ. Math. J.
  \textbf{25} (1976), 659--670.

\end{thebibliography}
\bibliographystyle{ijmart}

\end{document}